\documentclass{amsproc}

\usepackage{verbatim} 
\usepackage{amsmath}
\usepackage{amsfonts}
\usepackage{amssymb}
\usepackage{mathrsfs}
\usepackage{amsthm}
\usepackage{newlfont}
\usepackage{datetime}

\usepackage{fullpage,amsmath, amssymb,amscd}
\usepackage{latexsym, epsfig, color}
\usepackage[mathscr]{eucal}
\usepackage{xypic, graphicx}
\usepackage{amscd}
\usepackage[all, cmtip]{xy}
\usepackage{tikz,tikz-cd,xcolor}
\usepackage{tikz-cd}

\usepackage{scalerel,stackengine}
\stackMath
\newcommand\reallywidehat[1]{%
\savestack{\tmpbox}{\stretchto{%
  \scaleto{%
    \scalerel*[\widthof{\ensuremath{#1}}]{\kern-.6pt\bigwedge\kern-.6pt}%
    {\rule[-\textheight/2]{1ex}{\textheight}}
  }{\textheight}%
}{0.5ex}}%
\stackon[1pt]{#1}{\tmpbox}%
}
\usepackage{float, hyperref, nccmath, multirow, caption, subcaption}

\newcommand\cyr{%
 \renewcommand\rmdefault{wncyr}%
 \renewcommand\sfdefault{wncyss}%
 \renewcommand\encodingdefault{OT2}%
\normalfont\selectfont} \DeclareTextFontCommand{\textcyr}{\cyr}

\newtheorem{theorem}{Theorem}
\newtheorem{lemma}[theorem]{Lemma}
\newtheorem{corollary}[theorem]{Corollary}

\newtheorem{remark}[theorem]{Remark}

\def\Z{\mathbb Z}

\def\Q{\mathbb Q}
\def\R{\mathbb R}
\def\C{\mathbb C}
\def\F{\mathbb F}

\def\fp{\mathfrak p}

\def\Z{\mathbb Z}

\def\Q{\mathbb Q}
\def\R{\mathbb R}
\def\C{\mathbb C}
\def\F{\mathbb F}

\def\fp{\mathfrak p}

\def\Im{\operatorname{Im}}
\def\Re{\operatorname{Re}}

\def\det{\operatorname{det}}

\def\tr{\operatorname{tr}}

\def\Gal{\operatorname{Gal}}

\def\Frob{\operatorname{Frob}}

\def\id{\operatorname{id}}
\def\mod{\operatorname{mod}}

\def\disc{\operatorname{disc}}

\def\exp{\operatorname{exp}}

\def\gcd{\operatorname{gcd}}

\def\GL{\operatorname{GL}}

\def\li{\operatorname{li}}

\def\o{\operatorname{o}}

\def\log{\operatorname{log}}

\def\ds{\displaystyle}
\def\Tr{\operatorname{Tr}}
\def\Nr{\operatorname{N}}

\def\Im{\operatorname{Im}}
\def\Re{\operatorname{Re}}

\def\det{\operatorname{det}}

\def\tr{\operatorname{tr}}

\def\Gal{\operatorname{Gal}}

\def\Frob{\operatorname{Frob}}

\def\id{\operatorname{id}}
\def\mod{\operatorname{mod}}

\def\disc{\operatorname{disc}}

\def\exp{\operatorname{exp}}

\def\gcd{\operatorname{gcd}}

\def\GL{\operatorname{GL}}

\def\li{\operatorname{li}}

\def\o{\operatorname{o}}

\def\log{\operatorname{log}}

\def\ds{\displaystyle}
\def\Tr{\operatorname{Tr}}
\def\Nr{\operatorname{N}}

\begin{document}

\title{
Quantitative upper bounds related to an isogeny criterion for elliptic curves
}



\author{
Alina Carmen Cojocaru, Auden Hinz, and Tian Wang}
\address[Alina Carmen  Cojocaru]{
\begin{itemize}
\item[-]
Department of Mathematics, Statistics and Computer Science, University of Illinois at Chicago, 851 S Morgan St, 322
SEO, Chicago, 60607, IL, USA
\item[-]
Institute of Mathematics  ``Simion Stoilow'' of the Romanian Academy, 21 Calea Grivitei St, Bucharest, 010702,
Sector 1, Romania
\end{itemize}
} \email[Alina Carmen  Cojocaru]{cojocaru@uic.edu}

\address[Auden Hinz]{
\begin{itemize}
\item[-]
Department of Mathematics, Statistics and Computer Science, University of Illinois at Chicago, 851 S Morgan St, 1313
SEO, Chicago, 60607, IL, USA
\end{itemize}
} \email[Auden Hinz]{audenmh2@uic.edu}

\address[Tian Wang]{
\begin{itemize}
\item[-]
Department of Mathematics, Statistics and Computer Science, University of Illinois at Chicago, 851 S Morgan St, 1222
SEO, Chicago, 60607, IL, USA
\item[-]
Max Planck Institute for Mathematics,
Vivatsgasse 7, 53111 Bonn, Germany
\end{itemize}
} \email[Tian Wang]{twang213@uic.edu; twang@mpim-bonn.mpg.de}

\renewcommand{\thefootnote}{\fnsymbol{footnote}}
\renewcommand{\thefootnote}{\arabic{footnote}}

\renewcommand{\thefootnote}{\fnsymbol{footnote}}
\renewcommand{\thefootnote}{\arabic{footnote}}

\thanks{A.C.C. was partially supported  by a Collaboration Grant for Mathematicians from the Simons Foundation  
under Award No. 709008. }

\thanks{T.W. was partially supported by a University of Illinois at Chicago Dean’s Scholar Fellowship}


\begin{abstract}
For $E_1$ and $E_2$ elliptic curves defined over a number field $K$, without complex multiplication, 
we consider the function
${\cal{F}}_{E_1, E_2}(x)$
counting non-zero prime ideals
$\mathfrak{p}$
of the ring of integers of $K$, of good reduction for $E_1$ and $E_2$, 
of norm at most $x$,
and for which 
the Frobenius fields
$\Q(\pi_{\mathfrak{p}}(E_1))$ and  $\Q(\pi_{\mathfrak{p}}(E_2))$ are equal.
Motivated by an  isogeny criterion of Kulkarni, Patankar, and Rajan,
which states that
$E_1$ and $E_2$ are not potentially isogenous  
 if and only if
 ${\cal{F}}_{E_1, E_2}(x) = \o \left(\frac{x}{\log x}\right)$,
 we investigate the growth in $x$ of ${\cal{F}}_{E_1, E_2}(x)$.
We prove that
if  $E_1$ and $E_2$ are not potentially isogenous,
then
there exist positive constants
$\kappa(E_1, E_2, K)$,
$\kappa'(E_1, E_2, K)$,
and
$\kappa''(E_1, E_2, K)$
such that the following bounds hold:
(i) ${\cal{F}}_{E_1, E_2}(x) < \kappa(E_1, E_2, K) \frac{ x (\log\log x)^{\frac{1}{9}}}{ (\log x)^{\frac{19}{18}}}$;
(ii) ${\cal{F}}_{E_1, E_2}(x) < \kappa'(E_1, E_2, K) \frac{ x^{\frac{6}{7}}}{ (\log x)^{\frac{5}{7}}}$ 
under the Generalized Riemann Hypothesis for Dedekind zeta functions (GRH);
(iii) ${\cal{F}}_{E_1, E_2}(x) < \kappa''(E_1, E_2, K) x^{\frac{2}{3}} (\log x)^{\frac{1}{3}}$
under GRH, 
Artin's Holomorphy Conjecture for the Artin $L$-functions of number field extensions, 
and a  Pair Correlation Conjecture  for the zeros of the Artin $L$-functions of number field extensions. 
\end{abstract}

\maketitle


\section{Introduction}\label{Section-introduction}

Let $K$ be a number field, 
with ${\cal{O}}_K$ denoting its ring of integers 
and
$\overline{K}$ denoting a fixed algebraic closure.
In what follows, 
we use the letter $\mathfrak{p}$ to denote a non-zero prime ideal of ${\cal{O}}_K$ and refer to it as a {\it{prime}} of $K$,
 $\Nr_{K}(\mathfrak{p})$ to denote the norm of  $\mathfrak{p}$,
and 
$\F_{\mathfrak{p}}$ to denote the finite field ${\cal{O}}_K/\mathfrak{p}$.

Let $E_1$ and $E_2$ be 
 elliptic curves over $K$.
 We denote by $N_1$ and  $N_2$
 the norms  of the conductors  
 of $E_1$ and  $E_2$, respectively. 
 For a prime $\mathfrak{p}$ of $K$
 that is of good reduction for both $E_1$ and $E_2$
and
for each index $1 \leq j \leq 2$,  
we
consider 
the polynomial 
$P_{E_j, \mathfrak{p}}(X): = X^2 - a_{\mathfrak{p}}(E_j) X + \Nr_{K}(\mathfrak{p}) \in \Z[X]$,
 where $\Nr_{K}(\mathfrak{p}) + 1 -  a_{\mathfrak{p}}(E_j) $
 is the number  of $\F_{\mathfrak{p}}$-rational points of 
the reduction of $E_j$ modulo $\mathfrak{p}$.
We recall that,
for  any  rational prime  $\ell$ distinct from  the field characteristic of $\F_{\mathfrak{p}}$,
$P_{E_j, \mathfrak{p}}(X)$ is the characteristic polynomial of the image
$\rho_{E_j, \ell}(\Frob_{\mathfrak{p}})$ of a Frobenius element $\Frob_{\mathfrak{p}} \in \Gal \left(\overline{K}/K\right)$
under the $\ell$-adic Galois representation $\rho_{E_j, \ell}$ of $E_j$
defined by the action of $\Gal \left(\overline{K}/K\right)$
on the $\ell$-division points of $E_j\left(\overline{K}\right)$.
Viewing $P_{E_j, \mathfrak{p}}(X)$ in $\C[X]$ and denoting its roots by $\pi_{\mathfrak{p}}(E_j)$ and $\overline{\pi_{\mathfrak{p}}(E_j)}$, 
we recall that $\left| \pi_{\mathfrak{p}}(E_j)\right| = \sqrt{\Nr_{K}(\mathfrak{p})}$, 
which implies that $|a_{\mathfrak{p}}(E_j)| \leq 2 \sqrt{\Nr_{K}(\mathfrak{p})}$
and hence that $\Q(\pi_{\mathfrak{p}}(E_j))$ is either $\Q$ or  an imaginary quadratic field. 
In what follows, 
 we refer to $a_{\mathfrak{p}}(E_j)$ as the {\it{Frobenius trace}} and
to $\Q(\pi_{\fp\mathfrak{}}(E_j))$ as the {\it{Frobenius field}} associated to $E_j$ and $\mathfrak{p}$.

From now on, we assume that
 $E_1$ and $E_2$  are without complex multiplication.  Given a field extension $L$ of $K$, we say that $E_1$ and $E_2$ are  {\emph{$L$-isogenous}}  if there exists an isogeny from $E_1$ to $E_2$, defined over $L$.
We say that $E_1$ and $E_2$ are {\emph{potentially isogenous}}  if there exists a finite extension $L$ of $K$ such that
$E_1$ and $E_2$ are $L$-isogenous.
It is known that the following statements are equivalent:
$E_1$ and $E_2$ are potentially isogenous; 
$E_1$  and $E_2$ are $\overline{K}$-isogenous;
$E_1$ and $E_2$ are $L$-isogenous for some quadratic field extension $L$ of $K$;
either
$E_1$ and $E_2$ are $K$-isogenous, 
or
there exists a 
quadratic character $\chi$ such that 
$E_1$ 
and
the quadratic twist $E_2^{\chi}$ 
are
$K$-isogenous 
(e.g., see \cite[Lemma 3.1, p. 214; proof of Claim 3, p. 215]{LeFNa20}).
Our goal in this paper is to investigate questions arising from a criterion regarding whether $E_1$ and $E_2$
are potentially isogenous,
 as we explain below.

In \cite[Theorem 3, p. 90]{KuPaRa16}, Kulkarni, Patankar, and Rajan show that $E_1$ and $E_2$ are  potentially isogenous  
 if and only if the set of primes 
 $\mathfrak{p}$ of $K$, of good reduction for $E_1$ and $E_2$, such that
  $\Q(\pi_{\mathfrak{p}}(E_1)) = \Q(\pi_{\mathfrak{p}}(E_2))$,
has a positive upper density within the set of  primes of $K$, that is,
the counting function
\begin{equation*}\label{equal-Frob-fields-count}
{\cal{F}}_{E_1, E_2}(x) := 
\#\{
\mathfrak{p}: \Nr_K(\mathfrak{p}) \leq x, \Nr_K(\mathfrak{p}) \nmid N_1 N_2, 
\Q(\pi_{\mathfrak{p}}(E_1)) = \Q(\pi_{\mathfrak{p}}(E_2))
\}
\end{equation*}
satisfies
\begin{equation*}\label{equal-Frob-fields-ud}
\ds\limsup_{x \rightarrow \infty}
\frac{{\cal{F}}_{E_1, E_2}(x) }{\#\{\mathfrak{p}: \Nr_K(\mathfrak{p}) \leq x\}} > 0.
\end{equation*} 
Thus,
$E_1$ and $E_2$ are not potentially isogenous  
 if and only if
 ${\cal{F}}_{E_1, E_2}(x) = \o \left(\frac{x}{\log x}\right)$.
 
In relation to the above result, in \cite[Conjecture 1, p. 91]{KuPaRa16},
Kulkarni, Patankar, and Rajan
mention the following conjecture:
{\emph{
$E_1$ and $E_2$ are not potentially isogenous  if and only if there exists a positive constant $c(E_1, E_2, K)$,
 which  may depend on $E_1$, $E_2$,  and $K$,
 such that, for any sufficiently large $x$,
${\cal{F}}_{E_1, E_2}(x) < c(E_1, E_2, K) \frac{x^{\frac{1}{2}}}{\log x}$. 
}}
While the ``if" implication follows from the aforementioned result of Kulkarni, Patankar, and Rajan,
the ``only if" implication remains open and 
motivates the investigation of the growth of the function ${\cal{F}}_{E_1, E_2}(x)$.

In \cite[p. 1174]{CoFoMu05}, the authors record the following remark of Serre,
highlighting only the main idea of proof:  
if $E_1$ and $E_2$  are not potentially isogenous, then, under the Generalized Riemann Hypothesis for Dedekind zeta functions,
there exists a positive constant $c'(E_1, E_2, K)$, which depends on $E_1$, $E_2$,  and $K$,
such that, for any sufficiently large $x$,
\begin{eqnarray*}
&&
\#\{\mathfrak{p} \ \text{degree one prime}: \Nr_{K}(\mathfrak{p}) \leq x,  \Nr_K(\mathfrak{p}) \nmid N_1 N_2, 
\Q(\pi_{\mathfrak{p}}(E_1)) = \Q(\pi_{\mathfrak{p}}(E_2)) \not\in \{\Q(i), \Q(i \sqrt{3})\}\}
\\
&&
\hspace*{6cm}
 \leq c'(E_1, E_2, K) x^{\frac{11}{12}}.
\end{eqnarray*}
In \cite[Theorem 2, p. 43]{BaPa18}, Baier and Patankar address the growth of ${\cal{F}}_{E_1, E_2}(x)$ 
in the case $K = \Q$
and prove that,
under the Generalized Riemann Hypothesis for Dedekind zeta functions,
there exists a positive constant $c''(E_1, E_2)$, which depends on $E_1$ and $E_2$, 
such that, for any sufficiently large $x$,
$${\cal{F}}_{E_1, E_2}(x) < c''(E_1, E_2) x^{\frac{29}{30}} (\log x)^{\frac{1}{15}}.$$
In  \cite[Theorem 3, p. 43]{BaPa18}, Baier and Patankar also prove the following unconditional bound for
${\cal{F}}_{E_1, E_2}(x)$, resulting from an unconditional variation of the proof of their conditional result:
there exists a positive constant $c'''(E_1, E_2)$, which depends on $E_1$ and $E_2$, 
such that, for any sufficiently large $x$,
$${\cal{F}}_{E_1, E_2}(x) < c'''(E_1, E_2) \frac{x (\log \log x)^{\frac{22}{21}}}{(\log x)^{\frac{43}{42}}}.$$
The argument
highlighted
by Serre in \cite[p. 1174]{CoFoMu05}
 is based on a direct application of a conditional upper bound version of the Chebotarev density theorem
in the setting of an infinite 
Galois extension of $K$ defined by the $\ell$-adic Galois representations of $E_1$ and $E_2$, for a suitably chosen rational prime $\ell$. 
The proofs given by Baier and Patankar  in \cite{BaPa18} are based on indirect applications of conditional and unconditional effective  asymptotic versions of the Chebotarev density theorem, via the square sieve,
in the setting of a finite  Galois extension of $\Q$ defined by the residual modulo $\ell_1 \ell_2$  Galois representations of $E_1$ and $E_2$, for distinct suitably chosen rational primes $\ell_1$ and $\ell_2$.

The main goal of this paper is
to improve the current  upper bounds for ${\cal{F}}_{E_1, E_2}(x)$, as follows:
 unconditionally;
 under the Generalized Riemann Hypothesis for Dedekind zeta functions;
under the Generalized Riemann Hypothesis for Dedekind zeta functions,
Artin's Holomorphy Conjecture for the Artin $L$-functions of number field extensions, 
and
a  Pair Correlation Conjecture  regarding the zeros of the Artin $L$-functions of number field extensions. 
We shall refer to these latter hypotheses as GRH, AHC, and PCC,
and state them explicitly in the notation part of Section \ref{Section-introduction}.

 \begin{theorem}\label{theorem1}
Let $E_1$ and $E_2$ be elliptic curves over a number field $K$, 
  without complex multiplication,
  and not potentially isogenous.
 Denote by $N_1$ and $N_2$
 the norms of the conductors  of $E_1$ and $E_2$, respectively.
\begin{enumerate}
\item[(i)]
There exists a positive constant $\kappa(E_1, E_2, K)$, which depends on $E_1$, $E_2$, and $K$,
such that, 
for any sufficiently large $x$,
$${\cal{F}}_{E_1, E_2}(x) < \kappa(E_1, E_2, K) \frac{ x (\log\log x)^{\frac{1}{9}}}{ (\log x)^{\frac{19}{18}}}.$$
\item[(ii)]
If GRH  holds, then
there exists a positive constant $\kappa'(E_1, E_2, K)$, which depends on $E_1$, $E_2$, and $K$, 
 such that,
for any sufficiently large $x$,
$${\cal{F}}_{E_1, E_2}(x) < \kappa'(E_1, E_2, K) \frac{ x^{\frac{6}{7}}}{ (\log x)^{\frac{5}{7}}}.$$
\item[(iii)]
 If GRH, AHC, and PCC  hold, then
there exists a positive constant $\kappa''(E_1, E_2, K)$, which depends on $E_1$, $E_2$, and $K$, 
 such that,
for any sufficiently large $x$,
$${\cal{F}}_{E_1, E_2}(x) < \kappa''(E_1, E_2, K) x^{\frac{2}{3}} (\log x)^{\frac{1}{3}}.$$
\end{enumerate}
\end{theorem}

\begin{remark}\label{multiplicity-one}
{\emph{
Theorem \ref{theorem1} may be viewed under the general theme of
 strong multiplicity one results, such as those proven in  \cite{JaSh76}, \cite{MuPu17}, \cite{Ra94}, \cite{Ra00}, \cite{Wa14}, and \cite{Wo22}.
 In particular,  the methods developed in   \cite{MuPu17} and \cite{Wo22}
 are  applicable to bounding ${\cal{F}}_{E_1, E_2}(x)$ from above
  in the case $K = \Q$ and under hypotheses different from ours.
 Specifically, 
 letting $E_1$ and $E_2$ be elliptic curves over $\Q$, 
 without complex multiplication,
 not potentially isogenous, 
and 
assuming the Generalized Riemann Hypothesis 
for the Rankin-Selberg L-functions
associated to the symmetric power L-functions of $E_1$ and $E_2$,
the methods of  \cite{MuPu17} lead to 
${\cal{F}}_{E_1, E_2}(x) \leq \kappa_2(E_1, E_2) \frac{x^{\frac{7}{8}}}{(\log x)^{\frac{1}{2}}}$
(see \cite[Remark (ii), p. 567]{Wo22}),
while the methods of \cite{Wo22} lead to 
${\cal{F}}_{E_1, E_2}(x) \leq \kappa_3(E_1, E_2) \frac{x^{\frac{5}{6}}}{(\log x)^{\frac{1}{3}}}$ 
(see \cite[Theorem 1.11, p. 566]{Wo22}),
where
$\kappa_2(E_1, E_2)$
and
$\kappa_3(E_1, E_2)$
are positive constants that depend on $E_1$ and $E_2$.
It is not obvious if these methods generalize easily to tackle the case $K \neq \Q$ of
Theorem \ref{theorem1}
or to tackle the case $(\alpha_1, \alpha_2) \neq (\pm 1, \pm 1)$
of 
Theorem \ref{theorem3}.
}}
\end{remark}

An immediate application of Theorem \ref{theorem1} is another proof of the aforementioned isogeny criterion of Kulkarni, Patankar, and Rajan 
(see Section \ref{Section-isogeny-criterion}).

The proof of Theorem \ref{theorem1} relies on 
 upper bounds related to the Lang-Trotter Conjecture for Frobenius fields of one elliptic curve. We formulate the relevant results here for the  convenience of the reader.  
 Let $E/\Q$ be an elliptic curve  without complex multiplication and let $F$ be an imaginary quadratic field. 
    Lang and Trotter \cite{LaTr76} conjectured the asymptotic
    \begin{equation*}\label{Lang-Trotter Function}
   \pi_{E, F}(x) :=
    \#\{p \leq x: p\nmid N_E,  \Q(\pi_p(E))\simeq F\}  \sim C(E, F) \frac{x^{\frac{1}{2}}}{\log x},
    \end{equation*}
     where $C(E, F)$ is an explicit constant depending on $E$ and $F$.
  Zywina  \cite[Theorem 1.3, p. 236]{Zy15}  proved that unconditionally,   
  \[
  \pi_{E, F}(x)\leq \kappa_1(E, F)\frac{x(\log\log x)^2}{(\log x)^2},
  \]
 and that under GRH,
    \[
  \pi_{E, F}(x)\leq \kappa_1'(E, F)\frac{x^{\frac{4}{5}}}{(\log x)^{\frac{3}{5}}}.
  \]
Murty, Murty, and Wong \cite[Corollary 1.6, p. 406]{MuMuWo18} proved that under GRH, AHC, and PCC, 
    \[
  \pi_{E, F}(x)\leq \kappa_1''(E, F)\frac{x^{\frac{2}{3}}}{(\log x)^{\frac{1}{2}}}.
  \]
The proof  of Theorem \ref{theorem1} also relies on  the following  result which relates to the generalization of the
Lang-Trotter Conjecture on Frobenius traces formulated by  Chen, Jones, and Serban in \cite{ChJoSe22}.


\begin{theorem}\label{theorem3}
Let $E_1$ and $E_2$ be elliptic curves over a number field $K$, 
  without complex multiplication,
  and not potentially isogenous.
 Denote by $N_1$ and $N_2$
 the norms of the conductors  of $E_1$ and $E_2$, respectively.
 Let $\alpha_1$ and $\alpha_2$ be coprime integers, not both zero.
For $x > 0$, set
\begin{equation*}\label{Frob-traces-plus}
{\cal{T}}_{E_1, E_2}^{\alpha_1,\alpha_2}(x) 
:=
\#\{\fp: \Nr_K(\fp) \leq x, \gcd(\Nr_{K}(\fp), 6 N_1 N_2)=1,   \alpha_1a_\fp(E_1) + \alpha_2a_\fp(E_2) = 0\}.
\
\end{equation*}
\begin{enumerate}
\item[(i)]
There exists a positive constant $\kappa_0(E_1, E_2, K, \alpha_1, \alpha_2)$,
which depends on $E_1$, $E_2$,  $K$, $\alpha_1$, and $\alpha_2$,
such that, 
for any sufficiently large $x$,
$${\cal{T}}_{E_1, E_2}^{\alpha_1,\alpha_2}(x) < \kappa_0(E_1,E_2, K, \alpha_1, \alpha_2) \frac{ x (\log\log x)^{\frac{1}{9}}}{ (\log x)^{\frac{19}{18}}}.
$$
\item[(ii)]
If GRH  holds,  
then
there exists a positive constant $\kappa_0'(E_1, E_2, K, \alpha_1, \alpha_2)$,
which depends on $E_1$, $E_2$,  $K$, $\alpha_1$, and $\alpha_2$,
 such that, 
for any sufficiently large $x$,
$${\cal{T}}_{E_1, E_2}^{\alpha_1,\alpha_2}(x) < \kappa_0'(E_1, E_2, K, \alpha_1, \alpha_2) \frac{ x^{\frac{6}{7}}}{ (\log x)^{\frac{5}{7}}}.$$
\item[(iii)]
If GRH, AHC, and PCC hold,  
then
there exists a positive constant $\kappa_0''(E_1, E_2, K, \alpha_1, \alpha_2)$, 
which depends on $E_1$, $E_2$,   $K$, $\alpha_1$, and $\alpha_2$,
 such that, 
for any sufficiently large $x$,
$${\cal{T}}_{E_1, E_2}^{\alpha_1,\alpha_2}(x) < \kappa_0''(E_1, E_2, K, \alpha_1, \alpha_2)  x^{\frac{2}{3}} (\log x)^{\frac{1}{3}}.$$
\end{enumerate}
\end{theorem}

\begin{remark}\label{Mayle-Wang-constants}
{\emph{
Taking $K = \Q$ in the setting of Theorem \ref{theorem3}
and
 invoking \cite[Corollary 1.2, p. 3]{MaWa23} instead of \cite[Lemma 7.1, p. 409]{Lo16}
  in the proofs of parts (ii) and (iii),
 our proof leads to the following more explicit bounds.
 \begin{enumerate} 
\item[(ii')]
If GRH  holds,  
then
there exists an 
absolute,  effectively computable, positive  constant $\kappa_1'$
 such that, 
for any sufficiently large $x$,
$
{\cal{T}}_{E_1, E_2}^{\alpha_1,\alpha_2}(x) 
< 
\kappa_1' 
\frac{ x^{\frac{6}{7}}}{ (\log x)^{\frac{5}{7}}} 
\left(\log (N_1 N_2)\right)^{\frac{7}{2}}\left(\alpha_1\alpha_2\right)^{\frac{5}{2}}.
$
\item[(iii')]
If GRH, AHC, and PCC hold,  
then
there exists an
absolute,  effectively computable, positive  constant $\kappa_1''$
 such that, 
for any sufficiently large $x$,
$
{\cal{T}}_{E_1, E_2}^{\alpha_1,\alpha_2}(x) < \kappa_1''  x^{\frac{2}{3}} (\log x)^{\frac{1}{3}}  \left(\log (N_1 N_2)\right)^{\frac{3}{2}}\left(\alpha_1\alpha_2\right)^{\frac{1}{2}}. 
$
\end{enumerate}
}}
\end{remark}

\begin{remark}\label{lang-trotter-theorems}
{\emph{
Theorem \ref{theorem3} may  be viewed under the general Lang-Trotter theme of
 results 
 about the number of primes for which the Frobenius trace of an abelian variety is fixed, 
 such as those proven in 
 \cite{ChJoSe22},
 \cite{CoWa22}, \cite{CoWa23}, 
 \cite{Mu85}, \cite{MuMuSa88}, \cite{MuMuWo18}, \cite{Se81}, 
 \cite{ThZa18}, 
 and \cite{Zy15}.
 The connection between Theorem \ref{theorem3}, and thus   Theorem \ref{theorem1},
 with the Lang-Trotter Conjectures on Frobenius traces formulated in \cite[p. 33]{LaTr76} 
 and \cite[p. 382]{ChJoSe22}
  prompts the question of predicting,  conjecturally, the asymptotic behavior of 
${\cal{F}}_{E_1, E_2}(x)$
and
${\cal{T}}_{E_1, E_2}^{\alpha_1,\alpha_2}(x)$
for $E_1$, $E_2$, $\alpha_1$, and $\alpha_2$ as in the setting of Theorem \ref{theorem3}.
We relegate such investigations to a future project.
}}
\end{remark}


\medskip

\noindent
{\bf{Notation}}

\noindent






\noindent
$\bullet$
 Given a number field $K$, 
 we denote  
 by  ${\cal{O}}_K$ its ring of integers,
  by ${\sum}_K$ the set of non-zero prime ideals of  ${\cal{O}}_K$,
 by  $n_K$ 
 the degree of $K$ over $\Q$,
 by $d_K \in \Z \backslash \{0\}$ the discriminant of an integral basis of ${\cal{O}}_K$,
 and 
 by  $\disc(K/\Q) = \Z d_K \unlhd \Z$ the discriminant ideal of $K/\Q$.
 For a prime ideal $\mathfrak{p} \in {\sum}_K$, 
 we denote by $\Nr_{K}(\mathfrak{p})$ its norm in $K/\Q$.
 We say that $K$ satisfies the Generalized Riemann Hypothesis (GRH) if
 the Dedekind zeta function $\zeta_K$ of $K$ has the property that,
 for any $\rho \in \C$ with $0 \leq \Re \rho \leq 1$ and $\zeta_K(\rho) = 0$, we have $\Re(\rho) = \frac{1}{2}$. When $K=\Q$, the Dedekind zeta function is the Riemann zeta function, in which case we refer to GRH  as  the Riemann Hypothesis (RH). 

\noindent
$\bullet$
Given  a finite Galois extension $L/K$ of number fields
and a subset ${\cal{C}} \subseteq \Gal(L/K)$, 
stable under conjugation, 
we denote by $\pi_{ {\cal{C}} }(x, L/K)$
 the number of non-zero prime ideals of the ring of integers of $K$, unramified in $L$, of norm at most $x$,
 for which the Frobenius element is contained in $\cal{C}$.
We set
$$M(L/K)
 := 
 2 [L:K] |d_K|^{\frac{1}{n_K}} 
{\ds\prod_{p}}'
 p,$$
 with the dash on the product indicating that
 each of the primes $p$ therein lies over a non-zero prime ideal $\wp$ of ${\cal{O}}_L$,
 with  $\wp$ ramified in  $L$.

\noindent
$\bullet$
Given  a finite Galois extension $L/K$ of number fields
and  an irreducible character $\chi$ of the Galois group of $L/K$,
we denote 
by $\mathfrak{f}(\chi) \unlhd {\cal{O}}_K$ the global Artin conductor of $\chi$,
by $A_{\chi} := |d_L|^{\chi(1)} \Nr_{K}(\mathfrak{f}(\chi)) \in \Z$ the conductor of $\chi$,
and by ${\cal{A}}_{\chi}(T)$ the function of a positive real variable $T > 3$ defined by the relation
$$
\log {\cal{A}}_{\chi}(T) = \log A_{\chi} + \chi(1) n_K \log T.
$$
 
\noindent
$\bullet$
Given  a finite Galois extension $L/K$ of number fields,
we say that it satisfies Artin's Holomorphy Conjecture (AHC)
if, for any irreducible character $\chi$ of the Galois group of $L/K$,
the Artin L-function $L(s, \chi, L/K)$ extends to a function that is
analytic on  $\C$, except at $s = 1$ when $\chi = 1$.
We recall that, if
we assume GRH for $L$ and AHC for $L/K$, 
then, given 
any  irreducible character $\chi$ of the Galois group of $L/K$,
and given any non-trivial zero 
$\rho$ of $L(s, \chi, L/K)$, 
 the real part 
 of $\rho$ 
 satisfies
$\Re \rho = \frac{1}{2}$.
In this case, we write $\rho = \frac{1}{2} + i \gamma$, where 
$\gamma$ denotes the imaginary part 
of $\rho$.

\noindent
$\bullet$
Given  a finite Galois extension $L/K$ of number fields,
let us assume GRH for $L$ and AHC for $L/K$.
For an  irreducible character $\chi$ of the Galois group of $L/K$
and an arbitrary $T > 0$, 
we define the pair correlation function of $L(s, \chi, L/K)$ by
\begin{equation*}\label{pair-cor-fcn}
{\cal{P}}_T (X, \chi)
:=
\ds\sum_{- T \leq \gamma_1 \leq T}
\ds\sum_{- T \leq \gamma_2 \leq T}
w(\gamma_1 - \gamma_2)
e((\gamma_1 - \gamma_2) X),
\end{equation*}
where
$\gamma_1$ and $\gamma_2$ range over all the imaginary parts of the non-trivial zeroes
$\rho = \frac{1}{2} +  i \gamma$
 of $L(s, \chi, L/K)$,
counted with multiplicity, 
and where,
 for an arbitrary real number $u$,
$e(u) := \exp(2 \pi i u)$ and $w(u) := \frac{4}{4 + u^2}$.
We say that the extension $L/K$ satisfies the Pair Correlation Conjecture (PCC) 
if, 
 for any irreducible character $\chi$ of the Galois group of $L/K$ and for any $A > 0$ and $T > 3$,
 provided
$0 \leq Y \leq A \chi(1) n_K \log T$,
we have
$$
{\cal{P}}_T(Y, \chi) \ll_A \chi(1)^{-1} T \log {\cal{A}}_{\chi}(T).
$$

\medskip



\section{From shared Frobenius fields to shared absolute values of  Frobenius traces}\label{Section-Frob-fields-to-Frob-traces}

We keep the general setting and notation from Section \ref{Section-introduction}.
To prove Theorem \ref{theorem1}, we reduce the study of the primes $\mathfrak{p}$ for which the Frobenius fields of $E_1$ and $E_2$ coincide 
to a study of the primes $\mathfrak{p}$ for which the absolute values of the Frobenius traces of $E_1$ and $E_2$ coincide, as follows. 

\begin{lemma}\label{Frob-fields-to-traces}
Let $E_1$ and $E_2$ be elliptic curves over a number field $K$, 
non-isogenous over $K$.
 Denote by $N_1$ and $N_2$
 the norms of the conductors  of $E_1$ and $E_2$, respectively.
 Let $\mathfrak{p}$ be a degree one prime of $K$ such that 
 the rational prime
 $p := \Nr_{K}(\mathfrak{p})$ satisfies
 $p \nmid 6 N_1 N_2$.
 Assume that
 $\Q(\pi_{\mathfrak{p}}(E_1)), \Q(\pi_{\mathfrak{p}}(E_2)) \not\in \left\{\Q(i), \Q(i \sqrt{3})\right\}$.
Then 
$\Q(\pi_{\mathfrak{p}}(E_1)) = \Q(\pi_{\mathfrak{p}}(E_2))$
if and only if 
$|a_{\mathfrak{p}}(E_1)| = |a_{\mathfrak{p}}(E_2)|$.
\end{lemma}
\begin{proof}
The ``if'' implication is clear, since, for each $1 \leq j \leq 2$,  
$\Q(\pi_{\mathfrak{p}}(E_j)) = \Q\left(\sqrt{a_{\mathfrak{p}}(E_j)^2 - 4 p}\right)$. 
To justify the ``only if'' implication, 
we distinguish between 
$\mathfrak{p}$
supersingular 
and 
ordinary 
for 
 $E_1$ and $E_2$.
If $\mathfrak{p}$ is supersingular for both $E_1$ and $E_2$, then $a_{\mathfrak{p}}(E_1) = a_{\mathfrak{p}}(E_2) = 0$.
When $\mathfrak{p}$ is ordinary for both $E_1$ and $E_2$, or ordinary for one of $E_1$ or $E_2$, and supersingular for the other, 
we look at
the prime ideal factorization of $p$ in the ring of integers ${\cal{O}}_F$ of the imaginary quadratic field
$F := \Q(\pi_{\mathfrak{p}}(E_1)) = \Q(\pi_{\mathfrak{p}}(E_2))$.
By the lemma's hypothesis, the group of units 
of  ${\cal{O}}_F$ is
${\cal{O}}_F^{\times} = \{-1, 1\}$.
If $\mathfrak{p}$ is ordinary for both $E_1$ and $E_2$, 
then $p$ splits completely in $\Q(\pi_{\mathfrak{p}}(E_1))$ and $\Q(\pi_{\mathfrak{p}}(E_2))$, hence in $F$.
Then, as ideals in ${\cal{O}}_F$, either
$(\pi_{\mathfrak{p}}(E_1)) = (\pi_{\mathfrak{p}}(E_2))$,
or
$(\pi_{\mathfrak{p}}(E_1)) = \left(\overline{\pi_{\mathfrak{p}}(E_2)}\right)$. 
As such, 
$\pi_{\mathfrak{p}}(E_1) \in \left\{-\pi_{\mathfrak{p}}(E_2), \pi_{\mathfrak{p}}(E_2)\right\}$ 
or 
$\pi_{\mathfrak{p}}(E_1) \in \left\{-\overline{\pi_{\mathfrak{p}}(E_2)}, \overline{\pi_{\mathfrak{p}}(E_2)}\right\}$, 
which implies that
$\Tr_{F/\Q}(\pi_{\mathfrak{p}}(E_1)) \in \left\{- \Tr_{F/\Q}(\pi_{\mathfrak{p}}(E_2)), \Tr_{F/\Q}(\pi_{\mathfrak{p}}(E_2))\right\}$, where $\Tr_{F/\Q}(\alpha)$ denotes the trace of the algebraic number $\alpha\in F$. 
Since, for each $1 \leq j \leq 2$,  
$a_{\mathfrak{p}}(E_j) = \Tr_{\Q(\pi_{\mathfrak{p}}(E_j))/\Q}(\pi_{\mathfrak{p}}(E_j))$,
we obtain that $|a_{\mathfrak{p}}(E_1)| = |a_{\mathfrak{p}}(E_2)|$.
If $\mathfrak{p}$ is ordinary for one of $E_1$ or  $E_2$, say, for $E_1$, and supersingular for the other, say, for  $E_2$, 
then $p$ splits completely in $\Q(\pi_{\mathfrak{p}}(E_1))$ and ramifies in  $\Q(\pi_{\mathfrak{p}}(E_2))$,  contradicting
that 
$\Q(\pi_{\mathfrak{p}}(E_1)) = \Q(\pi_{\mathfrak{p}}(E_2))$.
Thus, this case does not occur.
\end{proof}

\section{Elliptic curves with shared absolute values of Frobenius traces}\label{Section-shared-traces}

We keep the general setting and notation from Section \ref{Section-introduction}.
In light of Lemma \ref{Frob-fields-to-traces}, in order to prove Theorem \ref{theorem1},
we  focus on the primes $\mathfrak{p}$ for which
$|a_{\mathfrak{p}}(E_1)| = |a_{\mathfrak{p}}(E_2)|$.
We view this condition as a combination of two  linear relations between the traces of $E_1$ and $E_2$,  namely
$a_{\mathfrak{p}}(E_1) + a_{\mathfrak{p}}(E_2) = 0$ and $a_{\mathfrak{p}}(E_1) - a_{\mathfrak{p}}(E_2) = 0$,
which are particular cases of   Theorem \ref{theorem3},.
Our goal in this section is to prove Theorem \ref{theorem3}.  

\subsection{Preliminaries}\label{Subsection-shared-traces-prelim}

We follow the methods developed in \cite{CoWa23} and \cite{Wa23}.
These methods already give rise to the  stated conditional estimates for ${\cal{T}}_{E_1, E_2}^{1,1}(x)$,
but need to be adjusted for the general conditional and unconditional bounds, as we explain below.

Consider the abelian surface
$$A := E_1 \times E_2.$$
For an arbitrary rational prime $\ell$,
consider 
 the residual  modulo $\ell$   Galois representations 
$\overline{\rho}_{A, \ell}$, 
$\overline{\rho}_{E_1, \ell}$, 
and
$\overline{\rho}_{E_2, \ell}$
of $A$, $E_1$, and $E_2$, respectively,
defined by the action of $\Gal\left(\overline{K}/K\right)$ on the $\ell$-division groups $A[\ell]$, $E_1[\ell]$,  and $E_2[\ell]$, respectively.
We recall that 
\begin{equation}\label{Galois-Image}
\overline{\rho}_{A, \ell}(\sigma) = (\overline{\rho}_{E_1, \ell}(\sigma), \overline{\rho}_{E_2, \ell}(\sigma)) \ \text{for any }  \sigma \in \Gal\left(\overline{K}/K\right),
\end{equation}
\begin{equation}\label{Galois-Trace}
\tr(\overline{\rho}_{E_j, \ell}(\Frob_\fp)) \equiv a_\fp(E_j) (\mod \ell) \ \text{for any } \fp \ \text{with}
\gcd(\Nr_K(\fp),  \ell N_j)=1  \text{and for any} \ 1 \leq j \leq 2.
\end{equation}
Setting
$$
G(\ell)
:= 
\left\{(M_1, M_2) \in \GL_2(\F_\ell) \times \GL_2(\F_\ell): \det M_1 = \det M_2\right\},
$$
we recall from \cite[Lemma 7.1, p. 409]{Lo16} that, 
thanks to our assumptions
that $E_1$ and $E_2$ are without complex multiplication and  not potentially isogenous,
there exists a positive integer  $c(A, K)$, which depends on $A$ and $K$, 
such that if $\ell > c(A, K)$, then
 $\Im \overline{\rho}_{A, \ell} = G(\ell)$, 
 that is, 
\begin{equation}\label{gal-div-field}
\Gal(K(A[\ell])/K) 
\simeq
G(\ell).
\end{equation}
For 
an arbitrary pair  of matrices $(M_1, M_2) \in G(\ell)$
and
for
each $1 \leq j \leq 2$, 
we denote by $\lambda_1(M_j), \lambda_2(M_j) \in \overline{\F}_{\ell}$ the eigenvalues of $M_j$.
Associated to $G(\ell)$, we consider the  groups
\begin{eqnarray*}
B(\ell) := 
\left\{
\left(\begin{pmatrix} 
  \ast  &\ast \\ 0 & \ast  
  \end{pmatrix},
  \begin{pmatrix} 
 \ast  & \ast \\ 0 & \ast
  \end{pmatrix}
  \right) 
  \in
  G(\ell) 
\right\},
&&
\Lambda(\ell) := 
\left\{
\left(\begin{pmatrix} 
  a  & 0\\ 0 & a  
  \end{pmatrix},
  \begin{pmatrix} 
  a  & 0\\ 0 & a  
  \end{pmatrix}
  \right) 
  \in
  G(\ell): a \in \F_\ell^{\times}
\right\},
\\
U(\ell) := 
\left\{
\left(\begin{pmatrix} 
  1  &\ast \\ 0 & 1 
  \end{pmatrix},
  \begin{pmatrix} 
  1  & \ast\\ 0 & 1  
  \end{pmatrix}
  \right) 
  \in
  G(\ell) 
\right\},
&&
 U'(\ell) := \Lambda(\ell) \cdot U(\ell),
\quad
P(\ell) := G(\ell)/\Lambda(\ell),
\end{eqnarray*}
and the sets
\begin{eqnarray*}
G(\ell)^{\#} 
&:= & 
\text{the set of conjugacy classes of } G(\ell),
\\
P(\ell)^{\#} 
&:= & 
\text{the set of conjugacy classes of } P(\ell),
\\
\cal{C}(\ell)^{\alpha_1,\alpha_2}
&:=&
\left\{
(M_1, M_2) \in G(\ell):
\lambda_1(M_j), \lambda_2(M_j) \in \F_{\ell}^{\times} \ \forall 1 \leq j \leq 2,
\alpha_1\tr M_1 + \alpha_2\tr M_2 = 0
\right\},
\\
\cal{C}_0(\ell)^{\alpha_1,\alpha_2}
&:=&
\left\{
(M_1, M_2) \in G(\ell):
\alpha_1\tr M_1 + \alpha_2\tr M_2 = 0
\right\},
\\
\cal{C}_{\text{Borel}}(\ell)^{\alpha_1,\alpha_2} 
&:=&
{\cal{C}}(\ell)^{\alpha_1,\alpha_2} \cap B(\ell),
\\
\widehat{\cal{C}}_{\text{Borel}}(\ell)^{\alpha_1,\alpha_2} 
&:=&
\; \text{the image of $\cal{C}_{\text{Borel}}(\ell)^{\alpha_1,\alpha_2}$ in  $B(\ell)/U'(\ell)$},
\\
\widehat{\cal{C}}_{\text{Proj}}(\ell)^{\alpha_1,\alpha_2} 
&:=&
\; \text{the image of $\cal{C}_0(\ell)^{\alpha_1,\alpha_2}$ in  $G(\ell)/\Lambda(\ell)$}.
\end{eqnarray*}



With the above notation, our strategy 
for proving parts (i) and (iii)  of Theorem \ref{theorem3}
is  to relate  
 \begin{center}
 ${\cal{T}}_{E_1, E_2}^{\alpha_1,\alpha_2}(x)$ to $\pi_{\widehat{\cal{C}}_{\text{Proj}}(\ell)^{\alpha_1,\alpha_2} }\left(x, K(A[\ell])^{\Lambda(\ell)}/K\right)$,
 \end{center}
 and
our strategy 
for proving part (ii)   of Theorem \ref{theorem3}
 is  to relate
\begin{center}
 ${\mathcal{T}}_{E_1, E_2}^{\alpha_1,\alpha_2}(x)$ to $\pi_{ {\cal{C}}_{\text{Borel}}(\ell)^{\alpha_1,\alpha_2}}\left(x, K(A[\ell])^{U'(\ell)}/K(A[\ell])^{B(\ell)}\right)$.
 \end{center}
After establishing these relations, we apply different variations of the effective Chebotarev density theorem
 to obtain upper bounds for 
 the number of primes $\mathfrak{p}$ whose Frobenius element satisfies the desired Chebotarev conditions.
In the end, we minimize the bounds by choosing 
$\ell$ suitably as a function of $x$.

Before executing this strategy,
we record a few properties of the groups and sets introduced above.

\begin{lemma}\label{lemma-subgroups}
For $\ell$ an arbitrary rational prime, the following statements hold.
\noindent
\begin{enumerate}
\item[(i)]
$\Lambda(\ell)$ is a normal subgroup of $G(\ell)$.
\item[(ii)]
$U'(\ell)$ is a normal subgroup of $B(\ell)$, with $B(\ell)/U'(\ell)$ an abelian group. 
\end{enumerate}
\end{lemma}
\begin{proof}
Part (i) is clear. 
Part (ii) is \cite[Lemma 11, p. 697]{CoWa23}. 
\end{proof}

\begin{lemma}\label{lemma-containment}
For $\ell$ an arbitrary rational prime, the following statements hold.
\noindent
\begin{enumerate}
\item[(i)] $U'(\ell) \ \cal{C}_{\text{Borel}}(\ell)^{\alpha_1,\alpha_2} \subseteq \cal{C}_{\text{Borel}} (\ell)^{\alpha_1,\alpha_2}$.
\item[(ii)]Every conjugacy class in $\cal{C}(\ell)^{\alpha_1,\alpha_2}$ contains an element of $B(\ell)$.
\item[(iii)]$\Lambda(\ell) \ \cal{C}_0(\ell)^{\alpha_1,\alpha_2} \subseteq \cal{C}_0(\ell)^{\alpha_1,\alpha_2}$.
\end{enumerate}
\end{lemma}
\begin{proof}
For part (i), the case $\alpha_1=\alpha_2=1$ is 
\cite[Lemma 14 (vi), p. 699]{CoWa23}.  In general, let $M'=(M_1',M_2')\in U'(\ell)$ be such that the diagonals are equal to some $a\in (\Z/\ell\Z)^{\times}$
and let $M=(M_1, M_2)\in  \cal{C}_{\text{Borel}}(\ell)^{\alpha_1,\alpha_2}$. Then $M'M\in B(\ell)$ and 
\[
\alpha_1\tr(M_1'M_1)+\alpha_2\tr(M_2'M_2)=a\left(\alpha_1\tr(M_1)+\alpha_2\tr(M_2)\right)=0.
\]
As such, $M'M\in  \cal{C}_{\text{Borel}} (\ell)^{\alpha_1,\alpha_2}$.

For part (ii), 
the  case $\alpha_1=\alpha_2=1$ is  \cite[ Lemma 16, p. 700]{CoWa23}.
In fact, by \cite[ Lemma 15, p. 700]{CoWa23}, every element in 
$\left\{
(M_1, M_2) \in G(\ell):
\lambda_1(M_j), \lambda_2(M_j) \in \F_{\ell}^{\times} \ \forall 1 \leq j \leq 2
\right\}$
is conjugate to an element in $B(\ell)$. In particular, every conjugacy class in $\cal{C}(\ell)^{\alpha_1,\alpha_2}$ contains an element of $B(\ell)$.

For part (iii), we provide a short proof.

Let $(aI, aI) \in\Lambda(\ell)$, with $a\in \F_\ell^{\times}$,  and let $M=(M_1,M_2)\in\cal{C}_0(\ell)^{\alpha_1,\alpha_2}$. 
Since $(aI, aI)$ and $M$ are in $G(\ell)$, the product $(aI, aI) M=(aM_1,aM_2)$ is in $G(\ell)$. 
Furthermore, since $M$ is in $\cal{C}_0(\ell)^{\alpha_1,\alpha_2}$, we have $\alpha_1\tr(M_1)+\alpha_2\tr(M_2)=0$, which implies that $\alpha_1\tr(aM_1)+\alpha_2\tr(aM_2)=0$. 
Therefore,  $(aI, aI)  M\in\cal{C}_0(\ell)^{\alpha_1,\alpha_2}$. 
\end{proof}

\begin{lemma}\label{lemma-sizes}
For $\ell$ an odd rational prime, the following statements hold.
\noindent
\begin{enumerate}
\item[(i)]
$|B(\ell)|  = (\ell-1)^3 \ell^2$.
\item[(ii)]
$\left|U'(\ell)\right| = (\ell - 1) \ell^2$.
\item[(iii)]
$|\Lambda(\ell)|=\ell-1$.
\item[(iv)]
$|P(\ell)|= (\ell-1)^2\ell^2(\ell+1)^2$.
\item[(v)]
$|G(\ell)^{\#}| \leq 4(\ell+1)^2(\ell-1)$
and 
$|P(\ell)^{\#}| \leq 16(\ell+1)^2$.
\end{enumerate}
\end{lemma}
\begin{proof}

Parts (i) and (ii) follow from \cite[Lemma 12,  pp. 697--698]{CoWa23}. Parts (iii) and (iv) are straightforward exercises derived from the definitions of the groups and the size of $\GL_2(\F_\ell)$. 
Part (v) is  \cite[Lemma 29, p. 45]{Wa23}, whose proof we include below.

The number of conjugacy classes of $\GL_2(\F_\ell)$ is $\ell^2-1$ (see \cite[p. 91]{FeFi60}).
Following \cite[p. 324]{JaLi01}, 
 these conjugacy classes can be classified into four types. 
 By considering each type, 
  we deduce that,
  for any $d\in \F_\ell^{\times}$,
  the number of conjugacy classes of $\GL_2(\F_\ell)$ with  determinant $d$ is at most $2\ell+2$. 
 Thus, $|G(\ell)^{\#}|\leq (2(\ell+1))^2 \cdot |\F_\ell^{\times}|=4(\ell+1)^2(\ell-1)$.
Now fix an arbitrary element 
 $\cal{C}\in G(\ell)^{\#}$. If there is an element $a\in \F_{\ell}^{\times}$ such that $(a I_2) \cal{C}=\cal{C}$, 
 then, by comparing  determinants, we obtain
$a^{4}=1$. So $a$ takes at most $4$ values in $\F_\ell^{\times}$.
 By the orbit-stabilizer theorem from group theory, each  $\Lambda(\ell)$-orbit 
 of $G(\ell)^{\#}$ contains at least 
 $ \frac{| \Lambda(\ell)|}{4}$ conjugacy classes. Therefore,
$|P(\ell)^{\#}|\leq  \frac{|G(\ell)^{\#}|}{|\F_{\ell}^{\times}|/4} \leq  16(\ell+1)^{2}$.
This completes the proof of (v).
\end{proof}

\begin{lemma}\label{lemma-sizes-sets}
For $\ell$ an odd rational prime such that 
$\ell$ does not divide at least one of $\alpha_1, \alpha_2$,
the following statements hold.
\noindent
\begin{enumerate}
\item[(i)]
$|\cal{C}_0(\ell)^{\alpha_1,\alpha_2}| \leq 2\ell^6 $.
\item[(ii)]
$|\widehat{\cal{C}}_{\text{Borel}}(\ell)^{\alpha_1,\alpha_2}| \leq 2 (\ell -1)$. 
\item[(iii)]
$|\widehat{\cal{C}}_{\text{Proj}}(\ell)^{\alpha_1,\alpha_2}|\leq  2 \ell^5  $.
\end{enumerate}
\end{lemma}
\begin{proof}
For parts (i), the case $\alpha_1=\alpha_2=1$ is 
\cite[Lemma 33, p. 51]{Wa23};
the general case is proved similarly, as we explain in what follows. 
We recall that,
for any $d\in \F_\ell^{\times}$ and  $t\in \F_\ell$,
 the number of matrices in $\GL_2(\F_\ell)$ with determinant $d$ and trace $t$ is $\ell\left(\ell+\left(\frac{t^2-4d}{\ell} \right)\right)$, where $\left(\frac{\cdot}{\ell}\right)$ denotes the Legendre symbol.
Therefore, 
\begin{eqnarray*}
|\cal{C}_0(\ell)^{\alpha_1,\alpha_2}| 
& =&
\ds \sum_{t\in \F_\ell} \sum_{d\in \F_\ell^{\times}}
\sum_{
M_1  \in \GL_2(\F_\ell)
\atop{\det M_1=d, \tr M_1=t}
}
\# \left\{
M_2\in \GL_2(\F_\ell): \det M_2=d, \tr M_2 =  - \alpha_2^{-1} \alpha_1 t (\mod \ell) 
\right\} 
\\
&\leq &
\ds 2 \sum_{t\in \F_\ell} \sum_{d\in \F_\ell^{\times}}
\sum_{
M_1 \in \GL_2(\F_\ell)
\atop{\det M_1=d, \tr M_1=t}
} \ell^2 
 \leq 2 \ell^{6},
\end{eqnarray*}
where $\alpha_2^{-1} (\mod \ell)$ is the inverse of $\alpha_2 (\mod \ell)$. 
Note that, since $\alpha_1$ and $\alpha_2$ are not both divisible by 
$\ell$, either  this inverse exists, or, if it does not,  the inverse of $\alpha_1 (\mod \ell)$ exists, in which case
 a similar argument  works using $\alpha_1^{-1} (\mod \ell)$.
This completes the proof of (i).

For part (ii), the case $\alpha_1=\alpha_2=1$ is \cite[Lemma 17, (iv), p.701]{CoWa23}
In the general case, we first consider the number of matrices in 
$
 \text{the image of $\cal{C}_{\text{Borel}}(\ell)^{\alpha_1,\alpha_2}$ in  $B(\ell)/U(\ell)\simeq T(\ell)$}.
$
They are clearly determined by the diagonal entries and can be counted as follows:
\begin{eqnarray*}
&&
\ds \sum_{a_1, a_2 \in  \F_\ell^{\times}}
\# \left\{
(b_1, b_2)  \in  \F_\ell^{\times}\times  \F_\ell^{\times}: b_1+b_2= - \alpha_2^{-1} \alpha_1(a_1+a_2) (\mod \ell), b_1b_2=a_1a_2 
\right\} 
\\
&\leq &
\ds 2 (\ell-1)^2,
\end{eqnarray*}
where $\alpha_2^{-1} (\mod \ell)$ is the inverse of $\alpha_2 (\mod \ell)$.  As before, if  the inverse does not exist, 
 a similar argument  works using $\alpha_1^{-1} (\mod \ell)$.
 
 Next, we observe that the inverse image of $\widehat{\cal{C}}_{\text{Borel}}(\ell)^{\alpha_1,\alpha_2}$ under the projection
 $B(\ell)/U(\ell) \to B(\ell)/U'(\ell)$ is  exactly $
 \text{the image of $\cal{C}_{\text{Borel}}(\ell)^{\alpha_1,\alpha_2}$ in  $B(\ell)/U(\ell)\simeq T(\ell)$}.
$
In all, 
\[
|\widehat{\cal{C}}_{\text{Borel}}(\ell)^{\alpha_1,\alpha_2}| \leq \frac{2 (\ell -1)^2}{|U'(\ell)/U(\ell)|}\leq 2 (\ell -1).
\]


Finally, from part (iii) of Lemma \ref{lemma-sizes} and part (i) of the current lemma, we deduce that
$|\widehat{\cal{C}}_{\text{Proj}}(\ell)^{\alpha_1,\alpha_2}| \leq \frac{|\cal{C}_0(\ell)^{\alpha_1,\alpha_2}|}{|\Lambda(\ell)|}\leq 2\ell^5$.
This completes the proof of  (iii).
\end{proof}
We are now ready to prove Theorem \ref{theorem3}.

\medskip 

\subsection{Proof of part (i) of Theorem \ref{theorem3}}\label{Subsection-shared-traces-proof-i}
The key ingredient is the unconditional effective Chebotarev density theorem
of Lagarias and Odlyzko \cite[Theorem 1.3, pp. 413--414]{LaOd77}, in the version 
 stated in \cite[Th\'{e}or\`{e}me 2, p. 132]{Se81}. 

We fix a rational prime $\ell $ such that $\ell > c(A, K)$
and 
such that $\ell$ does not divide at least one of $\alpha_1, \alpha_2$.
From (\ref{Galois-Image}) and (\ref{Galois-Trace}), we deduce that
\begin{equation}\label{second-relation-plus-GAP}
{\cal{T}}_{E_1, E_2}^{\alpha_1, \alpha_2}(x)
\leq
\pi_{\cal{C}_0(\ell)^{\alpha_1, \alpha_2} }\left(x, K(A[\ell])/K\right)+  n_K+\log M\left(K/\Q\right).
\end{equation}
In what follows, we bound from above the function on the right hand side of the inequality.

First, we relate
$\pi_{\cal{C}_0(\ell)^{\alpha_1, \alpha_2} }\left(x, K(A[\ell])/K\right)$
to 
 $\pi_{\widehat{\cal{C}}_{\text{Proj}}(\ell)^{\alpha_1, \alpha_2} }\left(x, K(A[\ell])^{\Lambda(\ell)}/K\right)$
by appealing to 
\cite[Proposition 7, p. 138, and Proposition 8 (b), p. 140]{Se81}), in the version stated in \cite[Corollary 5, p. 693]{CoWa23}.
Part (iii) of Lemma \ref{lemma-containment} 
ensures that we may apply 
these results
to 
the Galois group $G(\ell) = \Gal(K(A[\ell])/K)$,
 its normal subgroup $\Lambda(\ell) = \Gal(K(A[\ell])/K(A[\ell])^{\Lambda(\ell)})$,
 and
 the set
 $\cal{C}_0(\ell)^{\alpha_1, \alpha_2}$. 
  We deduce that
\begin{eqnarray*}\label{pi-versus-pi-tilde-upper-plus}
\pi_{
\cal{C}_0(\ell)^{\alpha_1, \alpha_2} 
}
\left(x, K(A[\ell])/K\right)
&\ll&
\pi_{
\widehat{\mathcal{C}}_{\text{Proj}}(\ell)^{\alpha_1, \alpha_2}
}
\left(x, K(A[\ell])^{\Lambda(\ell)}/K\right)
\\
&+&
n_K\left(\frac{x^{\frac{1}{2}}}{\log x} 
+ 
\log M(K(A[\ell])/K)
+
 \log M (K(A[\ell])^{\Lambda(\ell)}/K)\right).
\end{eqnarray*}

To bound $\log M\left(K(A[\ell])/K\right)$ and 
$\log M\left(K(A[\ell])^{\Lambda(\ell)}/K\right)$,
we proceed as in  \cite[(41), p. 708]{CoWa23}.
Specifically, 
relying on 
\cite[Proposition 6, p. 130]{Se81},
on  Lemma \ref{lemma-sizes},
 and
 on the  N\'eron–Ogg–Shafarevich criterion for abelian varieties,
we obtain that
 $$\log M\left(K(A[\ell])/K\right) \ll \frac{\log (\ell N_1 N_2 d_K)}{n_K},$$
 $$\log M\left(K(A[\ell])^{\Lambda(\ell)}/K\right) \ll \frac{\log (\ell N_1 N_2 d_K)}{n_K}.$$

To estimate
the counting function
$\pi_{
\widehat{\mathcal{C}}_{\text{Proj}}(\ell)^{\alpha_1, \alpha_2}
}
\left(x, K(A[\ell])^{\Lambda(\ell)}/K\right)$,
we apply \cite[Th\'{e}or\`{e}me 2, p. 132]{Se81} 
and obtain that
there exists an absolute, effectively computable,  positive constant $a_0$ such that,
 if 
 \begin{equation}\label{unconditional-restriction-ell}
 \log x > a_0 n_{K(A[\ell])^{\Lambda(\ell)}} \left(\log |d_{K(A[\ell])^{\Lambda(\ell)}}| \right)^2,
 \end{equation}
  then, for any $b>1$,
\begin{equation*}
\pi_{
\widehat{\mathcal{C}}_{\text{Proj}}(\ell)^{\alpha_1, \alpha_2}
}
\left(x, K(A[\ell])^{\Lambda(\ell)}/K\right) \ll_b  \frac{|\widehat{\mathcal{C}}_{\text{Proj}}(\ell)^{\alpha_1, \alpha_2}|}{|P(\ell)|} \li(x) 
+ \left|\left( \widehat{\mathcal{C}}_{\text{Proj}}(\ell)^{\alpha_1, \alpha_2} \right)^{\#}\right| \frac{x}{(\log x)^b},
\end{equation*}
where $\left( \widehat{\mathcal{C}}_{\text{Proj}}(\ell)^{\alpha_1, \alpha_2} \right)^{\#}$ 
is the set of conjugacy classes in $\widehat{\mathcal{C}}_{\text{Proj}}(\ell)^{\alpha_1, \alpha_2}$ 
and $\li(x) := \ds\int_{2}^x \frac{1}{\log t} \ d t$ is the logarithmic integral function.
Then, by Lemmas \ref{lemma-sizes} - \ref{lemma-sizes-sets},
we deduce that 
\begin{equation*}
\pi_{
\widehat{\mathcal{C}}_{\text{Proj}}(\ell)^{\alpha_1, \alpha_2}
}
\left(x, K(A[\ell])^{\Lambda(\ell)}/K\right) \ll_b \frac{x}{\ell \log x} +
\ell^2 \frac{x}{(\log x)^b}.
\end{equation*}

Finally,  we choose the prime $\ell=\ell(x)$  
such that $\ell > c(A, K)$,
such that
$\ell \nmid \alpha_1 \alpha_2$, 
or 
$\alpha_1  = 0$, $\ell \nmid \alpha_2$,
or
$\alpha_2  = 0$, $\ell \nmid \alpha_1$,
such that 
(\ref{unconditional-restriction-ell}) is satisfied,
and such that the final bounds are optimal, as follows.  

Once more relying on 
\cite[Proposition 5, p. 129]{Se81},
 Lemma \ref{lemma-sizes},
 and the  N\'eron–Ogg–Shafarevich criterion for abelian varieties,
we obtain that
\begin{align*}
n_{K(A[\ell])^{\Lambda(\ell)}}\left(\log |d_{K(A[\ell])^{\Lambda(\ell)}}| \right)^2 
& \leq |P(\ell)|n_K\left( (|P(\ell)|n_K-1)\log (\ell N_1N_2d_K) +(|P(\ell)|n_K-1)\log|P(\ell)n_K-1|  \right)^2\\
&\ll n_K^3 \ell^{18} (\log  (\ell N_1N_2d_K))^2.
\end{align*}

From  \cite[Lemma 7.1, p. 409]{Lo16},
we know that 
there exists an effectively computable, positive  constant $a(h_A, n_K)$, which depends on the Faltings height $h_A$ of $A$ and on  $n_K$,
such that,
  if $\ell>a(h_A, n_K)$,
 then (\ref{gal-div-field})
holds. 
Hence 
 condition
 (\ref{unconditional-restriction-ell})
 on $\ell$
 is
 ensured by the restrictions
\begin{equation*}\label{unconditional-restriction-ell-new}
a_1(h_A, n_K)
< 
\ell^{18} (\log \ell) ^2
< a_2(h_A, n_K, d_K, N_1, N_2) \log x
\end{equation*}
for some positive constants $a_1(h_A, n_K)$ and $a_2(h_A, n_K, d_K, N_1, N_2)$, which depend on $h_A$, $n_K$, $d_K$, $N_1$, and  $N_2$.
By taking $x > x_0(h_A, n_K, d_K,   N_1, N_2)$ for some positive real number which depends on $h_A$, $n_K$, $d_K$,  $N_1$, and $N_2$,
 we may choose the prime $\ell$ such that
\[
\ell(x)= 
\left[a_3 \frac{(\log x)^{\frac{1}{18}}}{(\log \log x)^{\frac{1}{9}}}\right]
\]
for some positive constant $a_3 = a_3(h_A, n_K, d_K, N_1, N_2, \alpha_1, \alpha_2)$, which depends on $h_A$, $n_K$,  $d_K$, $N_1$, $N_2$, $\alpha_1$, and $\alpha_2$.

Putting the bounds together, we deduce that 
\begin{equation*}\label{unknown-relation-plus}
{\cal{T}}_{E_1, E_2}^{\alpha_1, \alpha_2}(x) < \kappa_0(E_1, E_2, K, \alpha_1, \alpha_2) \frac{ x (\log\log x)^{\frac{1}{9}}}{ (\log x)^{\frac{19}{18}}} 
\end{equation*}
for some  positive constant $\kappa_0(E_1, E_2, K, \alpha_1, \alpha_2)$, which depends 
on $E_1$, $E_2$,  $K$, $\alpha_1$, and $\alpha_2$.
This completes the proof of part (i) of Theorem \ref{theorem3}.

\medskip

\subsection{Proof of part (iii) of Theorem \ref{theorem3}}\label{Subsection-shared-traces-proof-iii}

The key ingredient is
the conditional effective Chebotarev density theorem 
proved in
 \cite[Theorem 1.2, p. 402]{MuMuWo18},
which we use  in the reformulation stated in
  \cite[Theorem 7, p. 12]{CoWa22}).  

We fix a rational prime $\ell $ such that $\ell > c(A, K)$
and 
such that  $\ell$ does not divide at least one of $\alpha_1, \alpha_2$.
As in the proof of part (i),
after using
(\ref{second-relation-plus-GAP}), 
we focus our attention on estimating, from above,
$\pi_{\cal{C}_0(\ell)^{\alpha_1, \alpha_2} }\left(x, K(A[\ell])/K \right)$,
 this time under the assumptions of GRH, AHC, and PCC.

First, we proceed identically to part (i) and  deduce that
\begin{equation*}\label{pi-versus-pi-tilde-upper-plus-PCC}
\pi_{
\cal{C}_0(\ell)^{\alpha_1, \alpha_2} 
}
\left(x, K(A[\ell])/K\right)
\ll
\pi_{
\widehat{\mathcal{C}}_{\text{Proj}}(\ell)^{\alpha_1, \alpha_2}
}
\left(x, K(A[\ell])^{\Lambda(\ell)}/K \right)
+
n_K \left(\frac{x^{\frac{1}{2}}}{\log x} 
+ 
\frac{\log (\ell N_1 N_2 d_K)}{n_K}\right).
\end{equation*}
 Next, 
we apply 
\cite[Theorem 1.2, p. 402]{MuMuWo18} (which requires GRH, AHC, and PCC)
to estimate
the counting function
$\pi_{
\widehat{\mathcal{C}}_{\text{Proj}}(\ell)^{\alpha_1, \alpha_2}
}
\left(x, K(A[\ell])^{\Lambda(\ell)}/K\right)$.
By putting all estimates together, we deduce that 
\begin{eqnarray*}
\pi_{\cal{C}_0(\ell)^{\alpha_1, \alpha_2} }\left(x, K(A[\ell])/K\right) 
&\ll&
\frac{\left|\widehat{\cal{C}}_{\text{Proj}}(\ell)^{\alpha_1, \alpha_2} \right|}{|P(\ell)|} \cdot \frac{x}{\log x}
\\
&+&
n_K^{\frac{1}{2}}\left|\widehat{\cal{C}}_{\text{Proj}}(\ell)^{\alpha_1, \alpha_2} \right|^{\frac{1}{2}}
\left(\frac{\left| P(\ell)^{\#}\right|}{|P(\ell)|} \right)^{\frac{1}{2}}x^{\frac{1}{2}}
 \left(\frac{\log \left(\ell N_1 N_2 d_K\right)}{n_K}+\log  x\right)
\\
&+&
n_K\left(\frac{x^{\frac{1}{2}}}{\log x} 
+ 
\frac{\log (\ell N_1 N_2 d_K)}{n_K}\right).
\end{eqnarray*}
Then,  using 
Lemmas \ref{lemma-sizes} - \ref{lemma-sizes-sets},
we infer that
\begin{equation*}\label{GRH-AHC-PCC-plus}
\pi_{\cal{C}_0(\ell)^{\alpha_1, \alpha_2} }\left(x, K(A[\ell])/K\right) \ll
\frac{x}{\ell \log x}
+
n_K^{\frac{1}{2}} \ell^{\frac{1}{2}} 
x^{\frac{1}{2}} \left(\frac{\log (\ell N_1 N_2 d_K)}{n_K} +\log x\right).
\end{equation*}

Reasoning as in part (i), 
we may choose the prime $\ell$ such that
\begin{equation}\label{ell-x}
\ell(x) = \left[a_4 \frac{x^{\frac{1}{3}}}{(\log x)^{\frac{4}{3}}} \right]
\end{equation}
for some  positive constant $a_4 = a_4(h_A, n_K, d_K, N_1, N_2, \alpha_1, \alpha_2)$,
 which depends on $h_A$,  $n_K$,  $d_K$, $N_1, N_2$,
$\alpha_1$,  and $\alpha_2$.
Finally, recalling
(\ref{second-relation-plus-GAP}), 
we obtain that
\begin{equation*}\label{third-relation-plus}
{\cal{T}}_{E_1, E_2}^{\alpha_1, \alpha_2}(x)
\leq \kappa_0''(E_1, E_2, K, \alpha_1, \alpha_2) x^{\frac{2}{3}}(\log x)^{\frac{1}{3}}
\end{equation*}
for some  positive constant $\kappa_0''(E_1, E_2, K, \alpha_1, \alpha_2)$, which depends 
on $E_1$, $E_2$,  $K$, $\alpha_1$, and $\alpha_2$.
This completes the proof of part (iii) of Theorem \ref{theorem3}.

\medskip 
\subsection{Proof of part (ii) of Theorem \ref{theorem3}}\label{Subsection-shared-traces-proof-ii}

We base our proof on two key ingredients,
a modification of  \cite[Lemma 9, pp. 694--695]{CoWa23}
and
\cite[Theorem 2.3, p. 240]{Zy15},
as we explain below.

The first 
key ingredient is the following minor modification of   \cite[Lemma 9, pp. 694--695]{CoWa23},
which itself is  a generalization of 
 \cite[Lemma 4.4, p. 269]{MuMuSa88}. 
\begin{lemma}\label{max-lemma}
Let ${\cal{S}}$ be a non-empty set of prime ideals of $K$, 
let $(K_\fp)_{\fp \in {\cal{S}}}$ be a family of finite Galois extensions of $\Q$,
and
let $(\cal{C}_\fp)_{\fp\in \cal{S}}$ be a family of non-empty sets such that
 each $\cal{C}_\fp$  is a union  of conjugacy classes of $\Gal(K_\fp/\Q)$.
Assume that
there exist an absolute constant $c_1 > 0$
and 
a  function $f: \R \to (0, \infty)$
such that

\begin{equation}\label{hypothesis1}
\ds
n_{K_\fp}
\leq
c_1,
\end{equation}

\begin{equation}\label{hypothesis0}
\ds \log |d_{K_\fp}| 
\leq
 f(z)
 \text{ for all  $\fp$ such that } \Nr_K(\fp) \leq z.
\end{equation}
 For each 
 $x>2$, let $y = y(x) > 2$, $u =u(x) > 2$ be such that
\begin{equation}\label{hypothesis3}
u \leq y,
\end{equation}
and assume that, for any $\varepsilon > 0$,
\begin{equation}\label{hypothesis4}
u \geq 
c_2(\varepsilon) y^{\frac{1}{2}} (\log y)^{2 + \varepsilon} 
\
\text{for some constant} \ 
c_2(\varepsilon) > 0
\end{equation}
and
\begin{equation}\label{hypothesis5}
\ds\lim_{x \rightarrow \infty} \frac{f(x)}{(\log y)^{1 + \varepsilon}} = 0.
\end{equation}
Assume GRH for Dedekind zeta functions. 
Then, for any $\varepsilon > 0$,
there exists a constant $c(\varepsilon) > 0$ such that, for any sufficiently large $x$,
\begin{equation}\label{max-lemma-bound}
\#\left\{\fp: \Nr_K(\fp) \leq x, \fp \in {\cal{S}}\right\}
\leq
c(\varepsilon)
\ds\max_{
y \leq \ell \leq y + u
} 
\#\left\{\fp: \Nr_K(\fp) \leq x, \fp \in {\cal{S}}, \ell \nmid d_{K_\fp}, \left(\frac{K_\fp/\Q}{\ell} \right)\subseteq \cal{C}_\fp \right\}.
\end{equation}
\end{lemma}
 
 We apply this lemma to the set
${\cal{S}}^{\alpha_1, \alpha_2} := \{\fp:  \gcd(\Nr_K(\fp), 6 N_1 N_2)=1, \alpha_1 a_\fp(E_1) + \alpha_2 a_\fp(E_2) = 0\}$,
to the fields
$K_\fp := \Q(\pi_\fp(E_1), \pi_\fp(E_2))$, 
to the conjugacy classes
${\cal{C}}_\fp := \{\id_{K_\fp}\}$,
and 
to the function
$f(v) := 2\log (4v)$.
Note that, for ${\cal{S}}^{1, 1}$, this application is precisely the case  
$g = 2$ of \cite[Lemma 18, pp. 704--705]{CoWa23}. 
We obtain that, under the Riemann Hypothesis for the Riemann zeta function and GRH for the Dedekind zeta functions
of the number fields $K_\fp$, 
the following holds.

 For a fixed arbitrary $x > 2$, let $y := y(x)$ and $u := u(x)$ be real numbers such that
 $2 < u(x) < y(x)$. Assume that, 
 for any $\varepsilon > 0$,
  $\ds\lim_{x \rightarrow \infty} \frac{\log x}{(\log y(x))^{1 + \varepsilon}} = 0$
  and 
  there exists a positive constant $c'(\varepsilon)$ such that, for any sufficiently large $x$, 
 $u(x) \geq c'(\varepsilon) y(x)^{\frac{1}{2}} (\log y(x))^{2 + \varepsilon}$.
 Then, upon fixing an arbitrary $\varepsilon > 0$, there exist
 a positive constant $c(\varepsilon)$ and a positive real number $x_{\varepsilon}$ such that, for any $x \geq x_{\varepsilon}$
 and any $y(x)$, $u(x)$ satisfying the above conditions, we have
\begin{eqnarray*}\label{first-relation-plus}
{\cal{T}}_{E_1, E_2}^{\alpha_1, \alpha_2}(x)
\leq
c(\varepsilon)
\max_{
y \leq \ell \leq y + u}
\#\left\{
\fp: \Nr_K(\fp) \leq x, \gcd(\Nr_K(\fp), 6 N_1 N_2)=1, \alpha_1 a_\fp(E_1) + \alpha_2 a_\fp(E_2) = 0,
\right.
\\
&&
\hspace*{-4cm}
\left.
 \ell \ \text{splits completely in} \ K_\fp
\right\}.
\end{eqnarray*}
From (\ref{Galois-Image}), (\ref{Galois-Trace}),  and (\ref{gal-div-field}),
we deduce
that
\begin{equation}\label{second-relation-plus}
{\cal{T}}_{E_1, E_2}^{\alpha_1, \alpha_2}(x)
\leq
c(\varepsilon)
\max_{
y \leq \ell \leq y + u}
\pi_{ {\cal{C}}(\ell)^{\alpha_1, \alpha_2}  }\left(x, K(A[\ell])/K\right).
\end{equation}

The second key ingredient in our proof is
\cite[Theorem 2.3, p. 240]{Zy15} 
(see also its restatements \cite[Theorem 7, p. 693, and Corollary 8, p. 694]{CoWa23})). 
We will use this result to obtain upper bounds for the right hand side of  
(\ref{second-relation-plus}).

As in the proofs of parts (i) and (iii), we fix a rational prime $\ell$ such that $\ell > c(A, K)$
and 
such that $\ell$ does not divide at least one of $\alpha_1, \alpha_2$.

Parts (i) and (ii) of Lemma \ref{lemma-containment}  show that the hypotheses about
${\cal{C}}(\ell)^{\alpha_1, \alpha_2}$   
 needed to apply \cite[Theorem 2.3, p. 240]{Zy15}
 are satisfied.
Since $\Gal\left(K(A[\ell])^{U'(\ell)}/K(A[\ell])^{B(\ell)}\right) \simeq B(\ell)/U'(\ell)$ is abelian, 
AHC holds for the extension $K(A[\ell])^{U'(\ell)}/K(A[\ell])^{B(\ell)}$.
 Then, assuming GRH for the Dedekind zeta function of $K(A[\ell])^{U'(\ell)}$,
 by applying
 \cite[Theorem 2.3, p. 240]{Zy15},
 we obtain  that  
\begin{eqnarray*}
\pi_{ {\cal{C}}(\ell)^{\alpha_1, \alpha_2}  }\left(x, K(A[\ell])/K\right)
&\ll&
\frac{\left|\widehat{\cal{C}}_{\text{Borel}}(\ell)^{\alpha_1, \alpha_2}\right| \cdot |U'(\ell)|}{|B(\ell)|} \cdot \frac{x}{\log x}
\\
&+&
\left|\widehat{{\cal{C}}}_{\text{Borel}}(\ell)^{\alpha_1, \alpha_2}\right|^{\frac{1}{2}} [K(A[\ell])^{B(\ell)} :K] \frac{x^{\frac{1}{2}}}{\log x} \log M\left(K(A[\ell])^{U'(\ell)}/K(A[\ell])^{B(\ell)}\right)
\\
&+&
n_K\left(\frac{x^{\frac{1}{2}}}{\log x} + \log M\left(K(A[\ell])/K\right)\right)
\\
&+&
n_{K(A[\ell])^{B(\ell)}} \left(\frac{x^{\frac{1}{2}}}{\log x} 
+ \log M\left(K(A[\ell])^{U'(\ell)}/K(A[\ell])^{B(\ell)}\right)\right).
\end{eqnarray*}
To bound $|U'(\ell)|$ and $|B(\ell)|$, 
we use 
 Lemma \ref{lemma-sizes}.
To bound
$\left|\widehat{\cal{C}}_{\text{Borel}}(\ell)^{\alpha_1, \alpha_2}\right|$,
we  use
 Lemma \ref{lemma-sizes-sets}.
To bound 
$\log M\left(K(A[\ell])/K\right)$
and
$\log M\left(K(A[\ell])^{U'(\ell)}/K(A[\ell])^{B(\ell)}\right)$,
we proceed as in
parts (i) and (iii)
and obtain
\begin{center}
$\log M\left(K(A[\ell])/K\right) \ll \frac{\log (\ell N_1 N_2 d_K)}{n_K},$
$\log M\left(K(A[\ell])^{U'(\ell)}/K(A[\ell])^{B(\ell)}\right) \ll \frac{\log (\ell N_1 N_2 d_K)}{n_K}$.
\end{center}
Altogether, we deduce that
\begin{equation*}\label{GRH-AHC-plus}
\pi_{ {\cal{C}}(\ell)^{\alpha_1, \alpha_2}  }\left(x, K(A[\ell])/K \right)
\ll
\frac{x}{\ell \log x}
+
 \ell^{\frac{5}{2}} 
\frac{x^{\frac{1}{2}}}{\log x}\cdot  \frac{\log (\ell N_1 N_2 d_K)}{n_K}.
\end{equation*}

Now, we use
(\ref{second-relation-plus})
and infer that
\begin{equation*}\label{third-relation-plus}
{\cal{T}}_{E_1, E_2}^{\alpha_1, \alpha_2}(x)
\leq
c(\varepsilon)
\left(
\frac{x}{y(x) \log x}
+
 (y(x) + u(x))^{\frac{5}{2}} 
\frac{x^{\frac{1}{2}}}{\log x} \cdot \frac{\log ((y(x) + u(x)) N_1 N_2 d_K)}{n_K}
\right).
\end{equation*}

Finally, by invoking  \cite[Lemma 7.1, p. 409]{Lo16}  and recalling our constraints on $u(x)$ and $y(x)$, we choose 
$$y(x)= \left[a_5 \frac{x^{\frac{1}{7}}}{(\log x)^{\frac{2}{7}}}\right], \quad
u(x) =\left[a_6 y(x)^{\frac{1}{2}} (\log y(x))^{2 + \varepsilon}\right]$$
for some positive constants
 $a_5 = a_5(h_A, n_K, d_K, N_1, N_2, \alpha_1, \alpha_2)$ 
and $a_6 = a_6(h_A, n_K, d_K, N_1, N_2, \alpha_1, \alpha_2)$,
which  depend  on $h_A$, $n_K$, $d_K$, $N_1, N_2$, $\alpha_1$, and $\alpha_2$.
We deduce that
\begin{equation*}\label{fourth-relation-plus}
{\cal{T}}_{E_1, E_2}^{\alpha_1, \alpha_2}(x)
\leq
\kappa_0'(E_1, E_2, K, \alpha_1, \alpha_2) 
\frac{x^{\frac{6}{7}}}{(\log x)^{\frac{5}{7}}}
\end{equation*}
for some  positive constant $\kappa_0'(E_1, E_2, K, \alpha_1, \alpha_2)$ which depends 
on $E_1$, $E_2$,  $K$, $\alpha_1$, and $\alpha_2$.
This completes the proof of part (ii) of Theorem \ref{theorem3}.

\section{Elliptic curves with shared  Frobenius fields}\label{Section-shared-fields}

Let $E_1$ and $E_2$ be elliptic curves over a number field $K$, 
  without complex multiplication,
  and not potentially isogenous.
We keep the associated notation from the previous sections and prove Theorem \ref{theorem1}.

By Lemma \ref{Frob-fields-to-traces}, 
for any sufficiently large $x$, we have
\begin{eqnarray}\label{inequality-T-to-F}
&&
\hspace*{1.5cm}
{\cal{F}}_{E_1, E_2}(x)
\leq
{\cal{T}}_{E_1, E_2}^{1, 1}(x)
+
{\cal{T}}_{E_1, E_2}^{1, -1}(x)
\\
&&
\hspace*{0.5cm}
+
\ds\sum_{1 \leq j \leq 2}
\#\left\{
\fp:
\Nr_K(\fp) \leq x, \gcd(\Nr_{K}(\fp), 6 N_1 N_2)=1,
\fp \ \text{a degree one prime},
\Q(\pi_{\fp}(E_j)) \in \left\{\Q(i),  \Q\left(i \sqrt{3}\right)\right\}
\right\}
\nonumber
\\
&&
\hspace*{0.5cm}
+
\
\#
\left\{
\fp: 
\Nr_K(\fp) \leq x, 
\Nr_{K}(\fp) = p^f \ \text{for some rational prime} \ p \ \text{and some integer} \ f \geq 2
\right\}.
\nonumber
\end{eqnarray}
Note that the last term   is  bounded from above by 
$c x^{\frac{1}{2}}$ for some positive constant $c$ as explained in \cite[Proposition 7, p. 138]{Se81}.

\medskip
\noindent


\noindent
(i) For each of the first two terms  on the right hand side of inequality (\ref{inequality-T-to-F}), we invoke part (i) of Theorem \ref{theorem3}
and obtain the combined upper bound 
$\kappa_0(E_1, E_2, K) \frac{ x (\log\log x)^{\frac{1}{9}}}{ (\log x)^{\frac{19}{18}}}$
for some positive constant $\kappa_0(E_1, E_2, K)$, which depends on $E_1$, $E_2$, and $K$.
For each of the next two terms 
in the  sum over $1 \leq j \leq 2$
 on the right hand side of inequality (\ref{inequality-T-to-F}), we invoke a modification of 
 \cite[Theorem 1.3 (ii), p. 236]{Zy15} applied to the elliptic curve $E_j$ defined over $K$
  by counting only  degree one primes of norm at most $x$.
 This modification 
  relies on a variation of \cite[Lemma 5.1, p. 246]{Zy15} applied to $E_j$ defined over $K$  by counting only  degree one primes.
We obtain the upper bound 
$\kappa_1(E_j, K) \frac{x (\log \log x)^{2}}{(\log x)^2}$ for some positive constant
$\kappa_1(E_j, K)$, which depends on $E_j$ and $K$.
Putting everything together gives part (i) of Theorem \ref{theorem1}.


\medskip
\noindent
(ii) For each of the first two terms  on the right hand side of inequality (\ref{inequality-T-to-F}), we invoke part (ii) of Theorem \ref{theorem3}
and obtain the combined upper bound 
$2 \kappa_0'(E_1, E_2, K) \frac{x^{\frac{6}{7}}}{(\log x)^{\frac{5}{7}}}$ for some positive constant $\kappa_0'(E_1, E_2, K)$, which depends on $E_1$, $E_2$, and $K$.
For each of the next two terms 
in the  sum over $1 \leq j \leq 2$,  on the right hand side of inequality (\ref{inequality-T-to-F}), 
we invoke a  modification of  
\cite[Theorem 1.3 (i), p. 236]{Zy15} applied  to $E_j$  defined over $K$ by counting only  degree one primes,  as before.
We obtain the upper bound $\kappa_1'(E_j, K)\frac{x^{\frac{4}{5}}}{(\log x)^{\frac{3}{5}}}$
 for some positive constant
$\kappa_1'(E_j, K)$,  which depends on $E_j$ and $K$.
Putting everything together gives part (ii) of Theorem \ref{theorem1}.

\medskip
\noindent
(iii) For each of the first two terms  on the right hand side of inequality (\ref{inequality-T-to-F}), we invoke part (iii) of Theorem \ref{theorem3}
and obtain the combined upper bound 
$\kappa_0''(E_1, E_2, K) x^{\frac{2}{3}} (\log x)^{\frac{1}{3}}$
 for some positive constant $\kappa_0''(E_1, E_2, K)$, which depends on $E_1$, $E_2$, and $K$.
For each of the next two terms 
in the  sum over $1 \leq j \leq 2$ on the right hand side of inequality (\ref{inequality-T-to-F}), 
we invoke two modifications of \cite[Corollary 1.6, p. 406]{MuMuWo18} applied to $E_j$ defined over $K$
 by counting only  degree one primes.
The first modification is a variation of the proof ingredient  \cite[Lemma 15, p. 1548]{CoDa08}), 
which we make in order to work with an elliptic curve over $K$ and to count degree one primes of $K$.
The second modification is a variation of
 the argument in the proof of the second part of 
\cite[Corollary 1.6, p. 406]{MuMuWo18},
which we make in order
 to improve the resulting bound 
$x^{\frac{2}{3}} (\log x)^{\frac{1}{2}}$
to 
$x^{\frac{2}{3}} (\log x)^{\frac{1}{3}}$, as follows.
Letting $\ell(x)$ be as in (\ref{ell-x}),
instead of as in \cite[p. 422]{MuMuWo18}, we deduce that
each of the terms  on the right hand side of inequality (\ref{inequality-T-to-F}) is bounded from above by
$\kappa_1''(E_j, K) x^{\frac{2}{3}} (\log x)^{\frac{1}{3}}$
for some positive constant $\kappa_1''(E_j, K)$, which depends on $E_j$ and $K$.
Putting everything together gives part (iii) of Theorem \ref{theorem1}.

\section{Isogeny criterion for elliptic curves}\label{Section-isogeny-criterion}
As an immediate corollary of Theorem \ref{theorem1}, 
we deduce the following isogeny criterion of  Kulkarni, Patankar, and Rajan \cite[Theorem 3, p. 90]{KuPaRa16}.

\begin{corollary}\label{density}
    Let $E_1$ and $E_2$ be two elliptic curves over a number field $K$. Then $E_1$ and $E_2$ are potentially isogenous if and only if ${\cal{F}}_{E_1, E_2}(x)$ has a positive upper density within the set of  primes of $K$.
\end{corollary}
\begin{proof}
    For the``only if" implication, we assume that $E_1$ and $E_2$ are potentially isogenous. This implies that $E_1$ is isogenous over $K$ to a quadratic twist of $E_2$. Therefore, $|a_\fp(E_1)|=|a_\fp(E_2)|$ for all but finitely many primes $\fp$ of $K$. As in the ``if" implication of Lemma \ref{Frob-fields-to-traces}, we have that $\Q(\pi_{\fp}(E_1))=\Q(\pi_\fp(E_2))$ for all but finitely many primes $\fp$ of $K$. So ${\cal{F}}_{E_1, E_2}(x)$ has density one in the set of primes of $K$.
    
    For the ``if" implication, we prove the contrapositive. Assume that $E_1$ and $E_2$ are  not potentially isogenous. 
    Then, from part (i) of Theorem \ref{theorem1}, we deduce that ${\cal{F}}_{E_1, E_2}(x)$ is bounded from above by a set of density zero in the set of primes of $K$.
\end{proof}


\bigskip

{\small{

}

\end{document}